\documentclass[11pt]{amsart}

\usepackage{amsfonts,amsmath,amssymb,amsthm,amscd,amsxtra}
\usepackage{enumerate,epsfig,verbatim}

\usepackage{amssymb,amstext,amsmath,amscd,amsthm,amsfonts,enumerate,graphicx,latexsym}
\usepackage[usenames]{color}
\usepackage{tikz-cd}  
\usepackage{nicefrac}
\usepackage{array}

\usepackage{extpfeil}

\usepackage{comment}

\usepackage{mathptmx}

\usepackage{amsfonts,amsmath,amssymb,amsthm,amscd,amsxtra}
\usepackage{enumerate,verbatim}
\usepackage{centernot}
\usepackage{todonotes}

\usepackage[T1]{fontenc}
\usepackage[usenames,dvipsnames]{pstricks}
\usepackage[mathscr]{eucal}
\usepackage[notcite,notref]{} 
\usepackage[notcite, notref]{}
\usepackage[usenames,dvipsnames]{pstricks}
\usepackage[mathscr]{eucal}
\usepackage{amsfonts,amsmath,amssymb,amsthm,amscd,amsxtra}
\usepackage{enumerate,verbatim}
\usepackage[all,2cell,ps]{xy}
\usepackage{mathptmx}
\usepackage[notcite, notref]{}
\usepackage[pagebackref]{hyperref}
\usepackage{mathtools}
\usepackage{nicefrac}
\usepackage{array}
\usepackage{xcolor}
\usepackage[all,2cell,ps]{xy}
\usepackage{amsfonts}
\usepackage[margin=1.4in]{geometry}
\usepackage{color}
\usepackage[notcite,notref]{}
\usepackage{url}
\usepackage{datetime}
\usepackage{mathtools}
\usepackage[all,2cell,ps]{xy}
\usepackage{amssymb}
\usepackage{amsmath}
\usepackage{amsfonts}
\usepackage{amsmath}
\usepackage{amsthm}
\usepackage{amssymb}
\usepackage{amscd}
\usepackage{amsfonts}
\usepackage{amsxtra}     
\usepackage{epsfig}
\usepackage{verbatim}

\SelectTips{cm}{}

\def\latex/{{\protect\LaTeX}}
\def\latexe/{{\protect\LaTeXe}}
\def\amslatex/{{\protect\AmS-\protect\LaTeX}}
\def\tex/{{\protect\TeX}}
\def\amstex/{{\protect\AmS-\protect\TeX}}
\def\bibtex/{{Bib\protect\TeX}}
\def\makeindx/{\textit{MakeIndex}}

\usepackage{amsmath}
\usepackage{amsthm}
\usepackage{amssymb}
\usepackage{amscd}
\usepackage{amsxtra}     
\usepackage{epsfig}
\usepackage{verbatim}

\usepackage{enumerate}

\usepackage{amssymb,amstext,amsmath,amscd,amsthm,amsfonts,enumerate,graphicx,latexsym,color,stmaryrd}
\usepackage{cleveref}
\usepackage[all]{xy}
\usepackage{mathptmx}

\theoremstyle{plain}

\newtheorem{thm}{Theorem}[section]

\newtheorem{prop}[thm]{Proposition}
\newtheorem{cor}[thm]{Corollary}

\theoremstyle{definition}
\newtheorem{lem}[thm]{Lemma}

\newtheorem{eg}[thm]{Example}

\newtheorem{ques}[thm]{Question}

\newtheorem{rmk}[thm]{Remark}

\numberwithin{equation}{section}

\newcommand{\fm}{\mathfrak{m}}

\newcommand{\fp}{\mathfrak{p}}

\newcommand{\CC}{\mathbb{C}}

\newtheorem{chunk}[thm]{\hspace*{-1.065ex}\bf}

\DeclareMathOperator{\Max}{Max}

\DeclareMathOperator{\e}{e}

\def\GC-dim{\operatorname{\mathsf{G_C-dim}}}

\DeclareMathOperator{\pd}{pd}

\DeclareMathOperator{\id}{id}
\DeclareMathOperator{\Ann}{Ann}
\DeclareMathOperator{\G-dim}{G-dim}
\DeclareMathOperator{\depth}{depth}
\DeclareMathOperator{\Ext}{Ext}
\DeclareMathOperator{\Tor}{Tor}

\def\Tr{\mathsf{Tr}\hspace{0.01in}}

\DeclareMathOperator{\coker}{coker}

\DeclareMathOperator{\hh}{H}

\def\Hom{\operatorname{Hom}}

\DeclareMathOperator{\im}{im}

\DeclareMathOperator{\Soc}{Soc}
\DeclareMathOperator{\Supp}{Supp}
\DeclareMathOperator{\Spec}{Spec}

\def \syz {\Omega} 
\def \p {\mathfrak p}

\newcommand{\tensor}{\otimes^{\bf L}}

\def\Gc-dim{\operatorname{\mathsf{G_C-dim}}}

\def \gdim{\operatorname{G-dim}} 

\begin{document}
\baselineskip=15pt
	
\title[On the projective dimension of tensor products of modules
]{On the projective dimension of tensor products of modules}
	
\author{Olgur Celikbas}
\address{Olgur Celikbas\\ School of Mathematical and Data Sciences, West Virginia University, 
Morgantown, WV 26506 U.S.A}
\email{olgur.celikbas@math.wvu.edu}

\author{Souvik Dey}
\address{Souvik Dey\\
Faculty of Mathematics and Physics,
Department of Algebra,
Charles University, 
Sokolovsk\'{a} 83, 186 75 Praha, 
Czech Republic}
\email{souvik.dey@matfyz.cuni.cz}

\author{Toshinori Kobayashi}
\address{School of Science and Technology, Meiji University, 1-1-1 Higashi-Mita, Tama-ku, Kawasaki-shi, Kanagawa 214-8571, Japan}
\email{tkobayashi@meiji.ac.jp}

\thanks{Souvik Dey was partly supported by the Charles University Research Center program No.
UNCE/24/SCI/022 and a grant GACR 23-05148S from the Czech Science Foundation. Toshinori Kobayashi was partly supported by JSPS Grant-in-Aid for JSPS Fellows 21J00567}	
\maketitle{}

\vspace{-0.2in}
\begin{center}
\small{\emph{In memory of Nicholas Ryan Baeth}}
\end{center}

\begin{abstract} In this paper, we consider finitely generated modules over commutative Noetherian rings whose tensor products have finite projective dimension. We construct examples of modules of infinite projective dimension (and also of infinite Gorenstein dimension) whose tensor products nonetheless have finite projective dimension. Furthermore, we establish nontrivial conditions under which such examples cannot arise. For example, we prove that if the tensor product of two nonzero modules -- at least one of which is totally reflexive -- has finite projective dimension, then both modules in question must have finite projective dimension.
\end{abstract}


\section{Introduction}

Throughout $R$ denotes a commutative Noetherian ring and all $R$-modules are assumed to be finitely generated. If $R$ is assumed to be a local ring, then $\fm$ and $k$ denote the unique maximal ideal of $R$ and the residue field of $R$, respectively. We refer the reader to \cite{Av2, BH, Gdimbook} for any unexplained terminology in this paper and adopt the conventions that $\depth_R(0)=\infty$, $\dim_R(0)=-1$, and $\pd_R(0)=-\infty$; see \cite{Au, HW1}.

This paper originated from our discussions with Roger Wiegand, who informed us that the following question was raised at a commutative algebra meeting:

\begin{ques} \label{q1} Let $R$ be a commutative ring. If $M$ and $N$ are $R$-modules such that $\pd_R(M)<\infty$ and $\pd_R(N)<\infty$, then must $\pd_R(M\otimes_RN)<\infty$? What if $M=N$?
\end{ques}

It is not difficult to find counterexamples to Question \ref{q1}; however, there are also several special affirmative cases that are of interest to us. For example, we observed that, if $R$ is a $d$-dimensional Cohen-Macaulay local ring and $M$ is an $R$-module that is locally free on the punctured spectrum of $R$ such that $\pd_R(M)\leq d/2$, then $\pd_R(M\otimes_RM)<\infty$. One such example is the case where $R=k[\![x,y,z]\!]/(xy-z^2)$ and $M$ is the ideal of $R$ generated by $x$ and $y$; see \ref{cor3} and \ref{exCD} in the appendix for the details.

Wiegand \cite{RW} proved that, Question \ref{q1} has an affirmative answer if and only if the ring considered is regular or has depth zero. Wiegand's argument is inspiring to us; we state his theorem below and provide its proof in the appendix.

\begin{thm} [{Wiegand \cite{RW}}] \label{Roger} Let $R$ be a local ring. Then the following are equivalent:
\begin{enumerate}[\rm(i)]
\item $\depth(R)=0$ or $R$ is regular. 
\item For each $R$-module $M$ with $\pd_R(M)<\infty$, it follows that $\pd_R(M\otimes_RM)<\infty$.
\item For all cyclic $R$-modules $M$ and $N$ with $\pd_R(M)<\infty$ and $\pd_R(N)<\infty$, it follows that $\pd_R(M\otimes_RN)<\infty$.
\end{enumerate}
\end{thm}  

Wiegand \cite{RW}, motivated by Question \ref{q1} and Theorem \ref{Roger}, raised the following question:

\begin{ques} [{Wiegand \cite{RW}}] \label{soru1} Let $R$ be a commutative ring. If $M$ and $N$ are $R$-modules such that $\pd_R(M\otimes_RN)<\infty$, then must $\pd_R(M)<\infty$ or $\pd_R(N)<\infty$? 
\end{ques}

As in the case for Question \ref{q1}, affirmative answers to Question \ref{soru1} are known. For example, if $M\otimes_RN$ is nonzero and free, then $M$ and $N$ must be both projective; see \ref{Free}. Moreover, Question \ref{soru1} holds if the derived tensor product $M\tensor_R N$ of $M$ and $N$ is considered \cite[1.5.3(a)]{AFGD}. On the other hand, the question is not true in general, and one of our aims in this paper is to construct examples that provide a negative answer to Question \ref{soru1}. In fact, we show that the finiteness of the projective dimension of a nonzero tensor product $M\otimes_RN$ over a local ring $R$ does not necessarily imply the finiteness of the Gorenstein dimension (or polynomial growth of Betti numbers) of $M$ or $N$; see Examples \ref{Tos} and \ref{Tos2}. Additionally, we note that Question \ref{soru1} has a negative answer in the case where the projective dimension is replaced by the injective dimension; see Remark \ref{rmksemi}.

In addition to providing counterexamples, we point out some conditions under which Question \ref{soru1} is true; see the Appendix. Most of these conditions are immediate consequences of known results from the literature, but they motivate us to further investigate the projective dimension of tensor products and seek new, nontrivial conditions that imply Question \ref{soru1} holds. One such result, which we prove in this direction, is the following:

\begin{thm}\label{main1} Let $R$ be a ring and let $M$ and $N$ be nonzero $R$-modules. Assume:
\begin{enumerate}[\rm(a)]
\item $\Max \Spec(R)\subseteq \Supp_R(M)\cap \Supp_R(N)$ (for example, $R$ is local). 
\item $\pd_R(M\otimes_R N)\leq n$ for some $n\geq 0$.
\end{enumerate}
Then the following hold:
\begin{enumerate}[\rm(i)]
\item If $\Tor_i^R(M,N)=0$ for all $i=1, \ldots, n$, then $\pd_R(M)+ \pd_R(N)\leq n$.
\item If $\Ext^i_R(M,R)=0$ for all $i=1, \ldots, n$, then $M$ is projective and $\pd_R(N)<\infty$. 
\end{enumerate}
\end{thm} 

The proofs of the first and the second parts of Theorem \ref{main1} are entirely distinct, and are given in Section 3. The conclusion of Theorem \ref{main1}(i) is an extension of the following fact: If $M$ and $N$ are nonzero modules over a local ring $R$ such that $\pd_R(M\otimes_RN)\leq n$ for some $n\geq 0$ and $\Tor_i^R(M,N)=0$ for all $i\geq 1$, then $\pd_R(M)+\pd_R(N)\leq n$; see \ref{pdsum}(i). An immediate consequence of Theorem \ref{main1}(i) is the following; see also Corollary \ref{newcor1}.

\begin{cor} \label{corintronew1} If $M$ and $N$ are nonfree modules over a local ring $R$ such that $\pd_R(M\otimes_RN)=1$, then $\Tor_1^R(M,N)\neq 0$. 
\end{cor}

The vanishing hypothesis of Theorem \ref{main1}(ii) holds for totally reflexive modules: If $M$ is a totally reflexive $R$-module, then $\Ext^i_R(M,R)=0$ for all $i\geq 1$; see \cite[4.2.6]{Gdimbook}. Hence, the theorem yields the following result which is advertised in the abstract:

\begin{cor} \label{corintro1} Let $R$ be a  ring and let $M$ and $N$ be nonzero $R$-modules. Assume $M$ is totally reflexive. If $\pd_R(M\otimes_R N)<\infty$, then $M$ is projective and $\pd_R(N)<\infty$. 
\end{cor}

Note that Corollary \ref{corintro1} may fail if $M$ is not totally reflexive, even if $M$ has finite Gorenstein dimension; see Example \ref{Tos}. We also proved a result similar to Corollary \ref{corintro1} for Ulrich modules over Cohen-Macaulay local rings; see Corollary \ref{c37} and the paragraph preceding it.

A special, albeit important, case of our second main result concerning Question \ref{soru1} is the following theorem; see Theorem \ref{semimain} and \Cref{1.6prf}.

\begin{thm}\label{main2} Let $R$ be a Cohen-Macaulay local ring and let $M$ and $N$ be $R$-modules such that $M=\Omega_R L$ for some maximal Cohen-Macaulay $R$-module $L$. If $\pd_R(M\otimes_R N)<\infty$, then $M=0$ or $N=0$.
\end{thm}


The conclusion of Theorem \ref{main2} also holds if the projective dimension is replaced by the injective dimension; see Corollary \ref{corend}. However, it is worth noting that the conclusion of the theorem may fail if $L$ is not maximal Cohen-Macaulay, even if $M$ is a syzygy module. For example, if $R=k[\![x,y,z]\!]/(xy-z^2)$, $L=R/(x,y)$, and $M=N=\Omega_R L$, then $N=(x,y)R\neq 0$, but $\pd_R(M\otimes_RN)=2=\id_R(M\otimes_RN)$; see \cite[2.7]{CET} or Remark \ref{easy}.


\section{Counterexamples to Questions \ref{q1} and \ref{soru1}}

\subsection*{Some examples about Question \ref{q1}} In this subsection we construct examples that corroborate Theorem \ref{Roger} and give a negative answer to Question \ref{soru1}. 

Example \ref{exRoger} is due to Wiegand \cite{RW}. The first part of the example is included here to emphasize the fact that one cannot replace "or" with "and" in Question \ref{soru1}; see also Proposition \ref{cor3} concerning the second part of the example. 

\begin{eg} \label{exRoger} Let $R=k[\!|x,y]\!]/(xy)$ and let $M=R/(x+y)$.
\begin{enumerate}[\rm(i)]
\item Let $N=R/(x^2)$. Then it follows $M\otimes_RN\cong M$ so that $\pd_R(M\otimes_RN)=\pd_R(M)=1$ since $x+y$ is a non zero-divisor on $R$. On the  other hand, since $x^2$ is a zero-divisor on $R$, we see that $\pd_R(N)=\infty$. 
\item Let $N=R/(x-y)$. Then, since $x+y$ and $x-y$ are both non zero-divisors on $R$, it follows that $\pd_R(M)=1=\pd_R(N)$. Furthermore we have that $\pd_R(M\otimes_RN)=\infty$. 
\end{enumerate}
\end{eg}

We make use of the next lemma to obtain Example \ref{exYO} which gives a negative answer to Question \ref{soru1} over a ring that is not Cohen-Macaulay. 

\begin{lem} \label{YO} Let $R$ be a local ring such that $\depth(R)=1$. If $R$ is not regular, then there is an $R$-module $M$ such that $\pd_R(M)=1$ and $\pd_R\big(M^{\otimes n})=\infty$ for each $n\geq 2$.
\end{lem}

\begin{proof} Assume $R$ is not regular. Note that $\fm$ can be minimally generated by $\{x_1, \ldots, x_s\}$ for some elements $x_i$ of $R$, each of which is a non zero-divisor on $R$; see \ref{F1}. Therefore, by setting $M=R/(x_1)\oplus \cdots \oplus R/(x_s)$, we see that $\pd_R(M)=1$. 

Suppose $\pd_R\big(M^{\otimes n})<\infty$ for some $n\geq 2$. Then $k\cong R/(x_1, \ldots, x_s)$, being a direct summand of $M^{\otimes n}$, has finite projective dimension, which implies that $R$ is regular. Thus $\pd_R\big(M^{\otimes n})=\infty$ for each $n\geq 2$.
\end{proof}


\begin{eg} \label{exYO} Let $R=k[\![x,y]\!]/( xy)$ , where $k$ is a field of odd characteristic, and let $M=R/(x+y) \oplus R/(x-y)$. Then $\pd_R(M)=1$ and $\pd_R\big(M^{\otimes n})=\infty$ for each $n\geq 2$ as $\fm=(x+y , x-y)$ and $x+y, x-y$ are non-zero-divisors on $R$; see (the proof of) Lemma \ref{YO}. 
\end{eg}  

In Examples \ref{exRoger} and \ref{exYO}, the modules considered have projective dimension one. Next, in Example \ref{ex1}, we build on \cite[2.5]{CET}  and obtain an example of a tensor product of infinite projective dimension, where one of the modules in question has projective dimension three.

%
%
%
%
%
%

\begin{eg} \label{ex1} Let $R=k[\!|x,y,z,w]\!]/(xy)$ and let $\fp=(y,z,w)$. Then $R$ is a three-dimensional hypersurface and $\fp$ is a prime ideal of $R$. Set $M=R/(z, w, x+y)$, $N=\Tr(R/\fp)$ and $X=M\oplus N$. Then it follows that $\pd_R(M)=3$, $\pd_R(N)=1$, and $\pd_R(M\otimes_RN)=\infty$. Moreover we have $\pd_R(X)=3$ and $\pd_R(X\otimes_RX)=\infty$.

To establish these claims, first we note that $\{z, w, x+y\}$ is an $R$-regular sequence. So $\pd_R(M)=3$. Moreover, as $M$ is cyclic, we see that $\pd_R(M\otimes_RM)=3$.  Also, since $N$ is torsion-free, we conclude that $\pd_R(N\otimes_RN)=2$; see, for example, \ref{cor1}.  

There is a short exact sequence of the form:
$$ 0\to R \stackrel{ \begin{pmatrix} y  \\ z \\w \\ \end{pmatrix}} {\longrightarrow} R^{\oplus 3} \to N \to 0.$$
We obtain, by tensoring this short exact sequence with $M$, the following exact sequence of $R$-modules:
$$ M \stackrel{\begin{pmatrix} y  \\ 0 \\ 0 \\ \end{pmatrix}} {\longrightarrow} M^{\oplus 3} \to M\otimes_RN \to 0.$$
This exact sequence implies that $M\otimes_RN \cong M^{\oplus 2} \oplus (M/yM) \cong M^{\oplus 2} \oplus k$. Therefore we conclude that $\pd_R(M\otimes_RN)=\infty$.
Now we set $X=M\oplus N$. Then $\pd_R(X)=3$ and $\pd_R(X\otimes_RX)=\infty$ because $M\otimes_RN$ is a direct summand of $X\otimes_RX$. 
\end{eg}  

\subsection*{Some examples about Question \ref{soru1}} In this section we construct two examples giving a negative answer to Question \ref{soru1}; see also \cite[section 5]{Ce18} for some examples similar in flavor examining the finiteness of Gorenstein dimension of tensor products of modules. 

The following facts are used in Example \ref{Tos}. The first part is used in the proof of Theorem \ref{main1}(i), while the second part is required for the argument in Example \ref{Tos2}.

\begin{chunk} \label{pdsum} \label{GoodFact} Let $R$ be a local ring.
\begin{enumerate}[\rm(i)]
\item Let $M$ and $N$ be $R$-modules such that $\Tor_i^R(M,N)=0$ for all $i\geq 1$. It follows that $\pd_R(M\otimes_RN)=\pd_R(M)+\pd_R(N)$; see  \cite[A.7.6 and A.7.4.1]{Gdimbook} and \cite[1.1]{Mi}.
\item Let $0\neq x \in \fm$. Then $x$ is a non zero-divisor on $R$ if and only if $\pd_R(R/xR)<\infty$; see \cite[6.3]{AB58} and \cite[1.2.7(2)]{Av2}. 
\end{enumerate}
\end{chunk}

%

\begin{eg} \label{Tos} Let $R=k[\![x,y,z]\!]/(x^2)$, $M=R/(xy, z)$, and let $N=R/(xz, y)$. Then it follows that $\pd_R(M\otimes_RN)=2$ and
 $\pd_R(M)=\infty=\pd_R(N)$.

To establish these claims, first we note that $\{y,z\}$ is an $R$-regular sequence. Therefore, as $M\otimes_RN \cong R/(y,z)$, we conclude that $\pd_R(M\otimes_RN)=2$.

We have that $M=T_1 \otimes_R T_2$, where $T_1=R/(xy)$ and $T_2=R/(z)$. As $z$ is a non zero-divisor on both $R$ and $T_1$, we conclude that $\Tor_1^R(T_1, T_2)=0$. Furthermore, as $\pd_R(T_2)=1$, it follows that $\Tor_i^R(T_1, T_2)=0$ for all $i\geq 1$. Therefore, $\pd_R(T_1)+\pd_R(T_2)=\pd_R(T_1 \otimes_R T_2)=\pd_R(M)$; see \ref{pdsum}(i). As $\pd_R(T_2)<\infty$, we conclude that $\pd_R(M)=\infty$ if and only if $\pd_R(T_1)=\infty$. However, if $\pd_R(T_1)<\infty$ if and only $xy$ is a non zero-divisor on $R$; see \ref{GoodFact}(ii). As $x(xy)=0$, we see that $xy$ is a zero-divisor so that $\pd_R(T_1)=\infty$ and $\pd_R(M)=\infty$.

Note that $N=T_3 \otimes_R T_4$, where $T_3=R/(xz)$ and $T_4=R/(y)$. Furthermore, it follows that $y$ is a non zero-divisor on $R$ and $T_4$, and $xz$ is a zero-divisor on $R$. Therefore, by using a similar argument we used for $M$, we conclude that $\pd_R(N)=\infty$.
\end{eg}

\begin{eg} \label{Tos2} Let $R=k[\!|x,y,z,w]\!]/(x^2,xy,y^2)$, $M=R/(xw,z)$, and let $N=R/(xz,w)$. Then it follows that $\pd_R(M\otimes_RN)=2$, and $\G-dim_R(M)=\infty=\G-dim_R(N)$.

To see these, note that $M\otimes_RN=R/(xw,z,xz,w)=R/(z,w)$; hence $\pd_R(M\otimes_RN)=2$ since $\{z,w\}$ is an $R$-regular sequence.

We can write $M=T/zT$, where $T=R/(xw)$. Note that $z$ is a non zero-divisor on $T$. Hence $\G-dim_R(M)<\infty$ if and only if $\G-dim_R(T)<\infty$; see, for example, \cite[1.2.9]{Gdimbook}.

Next we note that $R$ is Golod since it is a non-Gorenstein ring which has codimension two; see \cite[(5.0.1) and 5.3.4]{Av2}. So, if $\G-dim_R(T)<\infty$, then it follows $\Ext^{i}_R(T,R)=0$ for all $i\gg 0$ \cite[1.2.7]{Gdimbook}, which implies that $\pd_R(T)<\infty$ and $xw$ is a non zero-divisor on $R$; see \ref{GoodFact}(ii). As $xw$ is a zero-divisor on $R$, we conclude that $\G-dim_R(T)=\infty$, that is, $\G-dim_R(M)=\infty$. Similarly we can observe that $\G-dim_R(N)=\infty$.

Finally, we note that $M$ and $N$ both have infinite complexity \cite{Av1}. More precisely, the Betti numbers of both $M$ and $N$ grow exponentially as these modules do not have finite projective dimension and $R$ is a Golod ring; see \cite[5.3.3(2)]{Av2} for more details.
\end{eg}

\begin{rmk} The conclusion of Example \ref{Tos2}, in relation to Question \ref{soru1}, is stronger than that of Example \ref{Tos}. Example \ref{Tos2} provides two modules, $M$ and $N$, over a local ring $R$ which is not Gorenstein, where both $M$ and $N$ have infinite Gorenstein dimension (as well as infinite complexity) and $M\otimes_RN$ has finite projective dimension; see also \cite[5.4]{Ce18}.
\end{rmk}

We conclude this section by pointing out that Question \ref{soru1} can fail if the projective dimension is replaced by the injective dimension.

\begin{rmk} \label{rmksemi} Let $R$ be a Cohen-Macaulay local ring with a canonical module $\omega$. Assume $R$ admits a nontrivial semidualizing module $C$, that is, $C\ncong \omega$ and $C\ncong R$; see \cite[2.3.2]{smd} for an example of such a ring $R$. It follows that $\id_R(C)=\infty$; see \cite[2.1.8 and 2.2.13]{smd}. Suppose $\id_R(C^{\dagger})<\infty$, where $(-)^{\dagger}=\Hom_R(-, \omega)$. Then $C^{\dagger} \cong \omega^{\oplus n}$ for some $n\geq 0$. This implies that $C\cong C^{\dagger\dagger} \cong R^{\oplus n}$, that is, $C \cong R$. Therefore, $\id_R(C^{\dagger})=\infty$. On the other hand, since $C\otimes_RC^{\dagger} \cong \omega$, we have that $\id_R(C\otimes_RC^{\dagger})<\infty$; see \cite[3.1.4 and 3.1.10]{smd}. 
\end{rmk}


\section{Proofs of Theorem \ref{main1}, Theorem \ref{main2}, and some corollaries}

\subsection*{Proof of Theorem \ref{main1}} This subsection is devoted to a proof of Theorem \ref{main1}. We start by proving the first part of the theorem:

\begin{proof} [A proof of Theorem \ref{main1}(i)] We assume $\pd_R(M\otimes_R N)\le n$ and $\Tor^R_{i}(M,N)=0$ for all $i=1, \ldots, n$ for some $n\geq 1$, and proceed to show that $\pd_R(M)+\pd_R(N)\leq n$. Recall that we assume $\Max \Spec(R)\subseteq \Supp(M)\cap \Supp(N)$. Thus, by localizing at maximal ideals of $R$, we may assume $R$ is local, and $M$ and $N$ are nonzero.

Let $X_{\bullet}= P_{\bullet} \otimes_R Q_{\bullet}$ be the total tensor product complex with differentials $\partial^{X_{\bullet}}_{\bullet}$, where $P_{\bullet}$ and $Q_{\bullet}$ are the minimal free resolutions of $M$ and $N$, respectively. Note that $\Tor_i^R(M,N)=\hh_i(X_{\bullet})$ for all $i\geq 0$. As $\Tor_i^R(M,N)=0$ for all $i=1, \ldots, n$, the following (minimal) complex is exact at $X_0, \ldots, X_n$:
\begin{equation}\tag{\ref{main1}.1}
 X_{n+1}\xrightarrow{\partial^{X_{\bullet}}_{n+1}}  X_n\xrightarrow{\partial^{X_{\bullet}}_{n}}  \cdots \to X_1 \xrightarrow{\partial^{X_{\bullet}}_1} X_0\to M\otimes_RN \to 0.
\end{equation}
Now, because $\Omega^{n+1}_R(M\otimes_RN)=\im(\partial^{X_{\bullet}}_{n+1})$ and $\pd_R(M\otimes_R N)\le n$, (\ref{main1}.1) implies that $\partial^{X_{\bullet}}_{n+1}=0$. This implies, by the definition of the differential map $\partial^{X_{\bullet}}_{\bullet}$ of the total complex, that $\im (\partial^P_i) \otimes_R Q_{n+1-i}=0$ and $P_{n+1-i}\otimes_R \im(\partial^Q_{i})=0$ for all $i=1,\dots,n+1$. Therefore $\im (\partial^P_{n+1}) \otimes_R Q_0=0$ and $P_0\otimes_R \im(\partial^Q_{n+1})=0$. As $M$ and $N$ are nonzero so are $P_0$ and $Q_0$. Consequently, it follows that $\im(\partial^P_{n+1})=0$ and $\im(\partial^Q_{n+1})=0$, and this implies $\pd_R(M)\le n$ and $\pd_R(N) \le n$. This implies that $\Tor_i^R(M,N)=0$ for all $i\geq 1$ since we already assume the vanishing of $\Tor_i^R(M,N)$ for all $i=1, \ldots, n$. Hence, we conclude from \ref{pdsum}(i) that $\pd_R(M)+\pd_R(N)=\pd_R(M\otimes_RN) \le n$.
\end{proof}

The following observation, a special case of \ref{cor1}(i), is from \cite[2.7]{CET}. We include a proof as the argument is straightforward and does not appeal to  \ref{cor1}(i).

\begin{rmk} \label{easy} Let $R$ be ring and let $M$ and $N$ be $R$-modules such that $\pd_R(M)=1$. 

There is an exact sequence $0\to G \to F \to M \to 0$ for some free $R$-modules $G$ and $F$. This implies that $\Tor_1^R(M,N)$ embeds into a finite direct sum of copies of $N$. Note that, since $\pd_R(M)<\infty$, $\Tor_1^R(M,N)$ is a torsion $R$-module. Thus, if $N$ is torsion-free, it follows that $\Tor_1^R(M,N)=0$, that is, $\Tor_i^R(M,N)=0$ for all $i\geq 1$. Therefore, if $N$ is torsion and $\pd_R(N)<\infty$, then $\pd_R(M\otimes_RN)<\infty$; see \ref{pdsum}.
\end{rmk}

In passing, we record some immediate consequences of Theorem \ref{main1}(i). 

\begin{cor} \label{newcor1} Let $R$ be a local ring, and let $M$ and $N$ be nonfree $R$-modules such that $\pd_R(M\otimes_RN)=1$. 
\begin{enumerate}[\rm(i)]
\item If $\depth_R(M)\geq \depth(R)-1$ and $N$ is torsion-free, then $\pd_R(M)=\infty$.
\item If $\depth(R)=1$ and $N$ is torsion-free, then $\pd_R(M)=\infty$.
\item If $\depth(R)=2$, and $M$ and $N$ are torsion-free, then $\pd_R(M)=\infty=\pd_R(N)$.
\end{enumerate}
\end{cor}

\begin{proof} It suffices to observe part (i). Suppose $\pd_R(M)<\infty$. Then, as $M$ is not free and $\depth_R(M)\geq \depth(R)-1$, the Auslander-Buchsbaum formula implies that $\pd_R(M)=1$.
So, Remark \ref{easy}  implies that $\Tor_1^R(M,N)=0$.  This is a contradiction in view of Corollary \ref{corintronew1}.
\end{proof}

The conclusion of Corollary \ref{newcor1}(ii) can fail if $N$ is not torsion-free: If $R=k[\!|x,y]\!]/(xy)$, $M=R/(x+y)$, and $N=R/(x^2)$, then $\depth(R)=1=\pd_R(M)$ and $N$ has torsion (the torsion submodule of $N$ is isomorphic to $k$); see Example \ref{exRoger}(i). Similarly, for Corollary \ref{newcor1}(iii) to hold, we need that the tensor product has projective dimension one; see the paragraph following Theorem \ref{main2}. Moreover, we do not know if there is an example as in part (iii) of Corollary \ref{newcor1}. Therefore, we ask:

\begin{ques} Are there nonfree torsion-free modules $M$ and $N$ over a local ring $R$ such that $\depth(R)=2$ and $\pd_R(M\otimes_RN)=1$?
\end{ques}

We proceed to prove the second part of Theorem \ref{main1}. In fact, we give three distinct proofs for Theorem \ref{main1}(ii), each of which seems to be of independent interest. Our first proof relies upon the next two lemmas:

\begin{lem}\label{f11} Let $R$ be a local ring and let $M$ be an $R$-module. If $\Ext^1_R\big(M, \Omega_R(M\otimes_RN)\big)=0$ for some nonzero $R$-module $N$, then $M$ is free.
\end{lem}

\begin{proof} It suffices to observe that $\Tor_1^R(\Tr M, M\otimes_R N)=0$, where $\Tr M$ is the Auslander transpose of $M$; see \cite[3.3(1)]{KOT}. As $\Tor_1^R(\Tr M, M\otimes_R N) \cong 
\underline{\Hom}_R(M, M\otimes_RN)$ \cite[3.9]{Yo}, we apply $\Hom_R(M, -)$ to the syzygy short exact sequence $0 \to \Omega_R(M\otimes_RN) \to F \to M\otimes_RN \to 0$, where $F$ is free, and see that 
the induced map $\Hom_R(M, F) \to \Hom_R(M , M\otimes_R N)$ is surjective. This implies that each $R$-module homomorphism $M \to M\otimes_R N$ factors through $F$, that is, 
$\underline{\Hom}_R(M, M\otimes_RN)=0$.
\end{proof}

We omit the proof of the following observation which can be proved by induction on $r$.

\begin{lem} \label{lm} Let $A$ and $B$ be $R$-modules and let $r\geq 0$ be an integer. Assume there is an exact sequence of $R$-modules
$0 \to A_r \to \cdots \to A_0 \to B \to 0$, where $A_i$ is a direct summand of a finite direct sum of copies of $A$ for all $i=0,\dots,r$. If $M$ is an $R$-module such that $\Ext^i_R(M, A)=0$ for all $i=1, \ldots, r+1$, then $\Ext^1_R(M,B)=0$.
\end{lem}

We make one more observation before proving \Cref{main1}(ii)

\begin{chunk}\label{new1} If $M$ and $N$ are $R$-modules such that $\Max \Spec(R)\subseteq \Supp(M)$, $M$ is projective, and $\pd_R(M\otimes_R N)<\infty$, then $\pd_R(N)<\infty$. Indeed, $\pd_{R_{\fm}} (N_{\fm}\otimes_{R_{\fm}}M_{\fm}) \le \pd_R(M\otimes_R N)$ for all $\fm \in \Max \Spec(R)$. So, for each $\fm \in \Max \Spec(R)$, it follows that $\pd_{R_{\fm}} (M_{\fm}) \le \pd_R(M\otimes_R N)$ since each $N_{\fm}$ is a non-zero free $R_{\fm}$-module.
\end{chunk}

We now give a proof of Theorem \ref{main1}(ii):

\begin{proof}[Proof of Theorem \ref{main1}(ii)] We have that $\Max \Spec(R)\subseteq \Supp(M)\cap \Supp(N)$. Thus, localizing at maximal ideals of $R$, we can assume $R$ is local, $M$ and $N$ are nonzero, and $\pd_R(M\otimes_R N)\le n$. Hence, it is enough to prove that $M$ is free; see \ref{new1}.  Note that $\pd_R\big(\Omega_R(M\otimes_R N)\big)\le n-1$, and hence $n-\pd_R(\Omega_R(M\otimes_R N))\geq 1$. Thus, Lemma \ref{lm} implies that $\Ext^1_R(M, \Omega_R(M\otimes_RN))=0$. Consequently, $M$ is free due to Lemma \ref{f11}.
\end{proof}

Next we will provide alternative proofs for the second part of Theorem \ref{main1}. Before we proceed to the proofs, we establish the following theorem.

\begin{thm} \label{injprop0} Let $R$ be a local ring and let $M$, $X$, and $Y$ be $R$-modules such that $X\neq 0$ and $\pd_R(X)=r<\infty$. Assume the following conditions hold:
\begin{enumerate}[\rm(i)]
\item There is a surjective $R$-module homomorphism $M\xtwoheadrightarrow{\alpha} X\otimes_RY$.
\item If $r\geq 1$, then $\Ext^i_R(M, Y)=0=\Tor_j^R(X,Y)$ for all $i=1, \ldots, r$ and all $j\geq 1$.
\end{enumerate}
Then there is a surjective $R$-module homomorphism $M \twoheadrightarrow Y$.
\end{thm}

\begin{proof} Suppose $r=0$. Then $X$ is free, and hence $X\otimes_R Y$ is isomorphic to a finite direct sum of copies of $Y$. Thus $X\otimes_R Y$ surjects onto $Y$. Consequently, $M$ surjects onto $Y$, as required. Hence, we may assume $r\geq 1$.

Let $0 \to R^{\oplus n_r} \to \cdots \to R^{\oplus n_0} \to X \to 0$ be a minimal free resolution of $X$. Note, since $X\neq 0$, we have that $n_0 \neq 0$. As $\Tor_i^R(X,Y)=0$ for all $i\geq 1$, we obtain the exact sequence $$0 \to Y^{\oplus n_r} \to \cdots \xrightarrow{\pi} Y^{\oplus n_0} \to X\otimes_RY \to 0,$$ where $\im(Y^{\oplus n_{i+1}} \to Y^{\oplus n_{i}}) \subseteq \fm Y^{\oplus n_{i}}$ for all $i=0, \ldots, r-1$. Then we consider the exact sequences:
$$0 \to Y^{\oplus n_r} \to \cdots \to Y^{\oplus n_1} \to \im(\pi) \to 0 \text{ and } 0 \to \im(\pi) \xrightarrow{f}  Y^{\oplus n_0} \xrightarrow{p}  X\otimes_RY \to 0.$$ 
We look at the following pullback (commutative) diagram:
\begin{center}
$\xymatrix{
0 \ar[r]& \im(\pi) \ar@{=}[d] \ar[r] & \exists \; P \ar[d] \ar[r] & M \ar[d]^{\alpha} \ar[r] &  0\\
0 \ar[r]& \im(\pi) \ar[r]^{f} & Y^{\oplus n_0} \ar[r]^{p\;\;\;\;\;\;\;} & X\otimes_RY \ar[r] & 0 
& &  & }$
\end{center}
As $\Ext^i_R(M, Y)=0$ for all $i=1, \ldots, r$, we use Lemma \ref{lm} with $A=Y$ and $B= \im(\pi)$ and conclude that $\Ext^1_R(M,\im(\pi))=0$. Hence the top short exact sequence in the diagram splits, that is, there is a splitting map $M \to P$. Hence taking the composition of this splitting map with the map $P\to Y^{\oplus n_0}$ in the diagram, we obtain an $R$-module homomorphism $\beta: M \to Y^{\oplus n_0}$ such that $\alpha=p \beta$.

Note that, since $\im(f)\subseteq \fm Y^{\oplus n_0}$, it follows that the map $f\otimes 1_k: \im(\pi) \otimes_Rk \to Y^{\oplus n_0}\otimes_Rk$ is the zero map. Therefore $0=\im(f\otimes 1_k)=\ker(p\otimes 1_k)$, that is, $p\otimes 1_k$ is an isomorphism. This implies that $\beta \otimes 1_k$ is surjective since $\alpha\otimes 1_k= (p\otimes 1_k) \circ (\beta \otimes 1_k)$. Consequently, by Nakayama's lemma, $\beta$ is also surjective. So, we obtain a surjection $M \twoheadrightarrow Y^{\oplus n_0}$, as claimed.
\end{proof}

We assemble some basic results which play an important role in the sequel:

\begin{rmk} \label{rmk1} Let $R$ be a ring and let $0 \to A \xrightarrow[]{f} B \xrightarrow[]{g} C \to 0$ be a short exact sequence of $R$-modules. Consider the syzygy exact sequence $0 \to \Omega_R C \to G \xrightarrow[]{\pi}  C \to 0$, where $G$ is a free $R$-module. 
Then, by taking the pullback of the maps $g$ and $\pi$, we obtain an $R$-module homomorphism $h:G \to B$ and a short exact sequence of $R$-modules of the form
$$0 \to \syz_R C \to A\oplus G \xrightarrow[]{[f,\; h]} B  \to 0,$$
where $[f,\; h](x,y)=f(x)+h(y)$ for all $x \in A$ and $y\in G$. 
\end{rmk}


\begin{rmk} \label{rmk2} Let $R$ be a local ring and let $f:X \to Y$ be an $R$-module homomorphism. Assume $X$ has no nonzero free summand and $Y$ is free. Then it follows that $\im(f) \subseteq \fm Y$. 

To establish this, we set $Y=R^{\oplus n}$ for some $n\geq 0$. Then $f=i_1p_1f+ \ldots+ i_np_nf$, where $p_j:Y \to R$ and $i_j:R \to Y$ are the natural projections and injections, respectively. Suppose $\im(p_jf)\nsubseteq \fm$ for some $j$ with $1 \leq j \leq n$. Then $\im(p_jf)=R$ and this gives a surjective $R$-module homomorphism $p_jf: X \twoheadrightarrow R$, which shows that $R$ is a direct summand of $X$. Therefore $\im(p_jf)\subseteq \fm$ for each $j$, and this implies that $\im(f)\subseteq i_1(\fm)+\cdots+i_n(\fm)\subseteq \fm Y$.  \qed
\end{rmk}

We are now ready to provide two more proofs for Theorem \ref{main1}(ii). The first proof we provide relies upon on Theorem \ref{injprop0}, while the second one does not.

\begin{proof}[Alternative proofs of Theorem \ref{main1}(ii)] Recall we assume $\Max \Spec(R)\subseteq \Supp(M)\cap \Supp(N)$. Therefore, localizing at maximal ideals of $R$, it is enough to assume $R$ is local, $M$ and $N$ are nonzero, and $\pd_R(M\otimes_R N)\le n$, and prove that $M$ is free.

We write $M=H \oplus P$ for some free $R$-module $P$ and an $R$-module $H$ which has no nonzero free summand. Hence, it suffices to show $H=0$. Note also that, since $H\otimes_R N$ is a direct summand of $M\otimes N$, it follows that $\pd_R(H\otimes_R N)\leq n$.

(First proof) Suppose $H\neq 0$ and seek a contradiction. Set $r=\mu(N)$. Then there is a surjection $H^{\oplus r}\to H\otimes_R N$. Moreover, we have that $\Ext^{i}_R(H,R)=0$ for all $i=1, \ldots, n$. Hence we use Theorem \ref{injprop0} by setting $X=H\otimes_R N$ and $Y=R$, and obtain a surjection $H^{\oplus r} \twoheadrightarrow  R$. This implies that $R$ is a direct summand of $H^{\oplus r}$. Then $R$ is also a direct summand of $H$; see, for example, \cite[1.2]{BLW}. Thus $H=0$, as required.

(Second proof) We consider the syzygy exact sequence $0 \to \syz_R N \xrightarrow[]{\alpha} F \xrightarrow[]{\beta} N \to 0$, where $F$ is a free $R$-module. 
We tensor this short exact sequence with $H$ and obtain the exact sequence: 
\begin{equation}\tag{\ref{rmk1}.1}
0 \to \ker(1_{M'}\otimes \beta) \xrightarrow[]{\gamma} H\otimes F \xrightarrow[]{1_{H}\otimes \beta} H\otimes N \to 0.
\end{equation}
In view of Remark \ref{rmk1}, the exact sequence (\ref{rmk1}.1) yields an exact sequence of the form
\begin{equation}\tag{\ref{rmk1}.2}
0 \to \syz_R(H\otimes_R N) \to \ker(1_{H}\otimes \beta) \oplus G\xrightarrow[]{[\gamma,\delta]} H\otimes_R F  \to 0, 
\end{equation}
where $G$ is a free $R$-module and $\delta: G \to H\otimes_R F$ is an $R$-module homomorphism.

Note that, since $H\otimes_R N$ is a direct summand of $M\otimes_R N$, it follows that $\Ext_R^i(H\otimes_R F,R)$ is a direct summand of $\Ext_R^i(M\otimes_R F,R)$ for all $i=1,\dots, n$. Hence, we have
\begin{equation}\tag{\ref{rmk1}.3}
\Ext_R^i(H \otimes_R F, R)=0 \text{ for all } i=1,\dots, n.
\end{equation}

As $\pd_R(M\otimes_R N)\le n$, it follows that $\pd_R\big(\Omega_R(M'\otimes_R N)\big)<n$. Set $A=R$, $B=\Omega_R(H\otimes_R N)$, and $r=\pd_R(B)$. Then we conclude from (\ref{rmk1}.3) and Lemma \ref{lm} that $\Ext_R^1(H\otimes_R F, B)$ vanishes. So, the exact sequence (\ref{rmk1}.2) splits and yields a splitting map $$[\epsilon,\phi]^T\colon H\otimes_R F\to \ker(1_{H}\otimes \beta) \oplus G,$$ where $[\gamma,\delta]\circ [\epsilon,\phi]^T=1_{H\otimes_RF}$, that is, $\gamma\epsilon+\delta\phi =1_{H\otimes_RF}$. Therefore, we have:
\begin{equation}\tag{\ref{rmk1}.4}
H\otimes_R F=\im(1_{M'\otimes_RF})=\im(\gamma\epsilon+\delta\phi) \subseteq \im(\gamma\epsilon)+\im(\delta\phi) \subseteq \im(\gamma\epsilon)+ \fm (M'\otimes_RF).
\end{equation}
Here, in (\ref{rmk1}.4), the last inclusion holds due to the fact stated in \ref{rmk2}: as $\phi:H\otimes_R F \to G$ is an $R$-module homomorphism, $H\otimes_R F$ has no nonzero free summand, and $G$ is free, it follows from \ref{rmk2} that $\im(\phi) \subseteq \fm G$.
Now we use Nakayama's lemma and deduce from (\ref{rmk1}.4) that $\im(\gamma\epsilon)=H\otimes_RF$. So, $\gamma\epsilon: M'\otimes_RF \to M'\otimes_RF$, being a surjective map, is injective, that is, $\gamma\epsilon$ is bijective. Thus, $\gamma$ is surjective and $0=\coker(\gamma) \cong H\otimes_RN$; see (\ref{rmk1}.2). This forces $H=0$ since $N\neq 0$.
\end{proof}


\subsection*{A proof of Theorem \ref{main2}} This subsection is devoted to the proof of Theorem \ref{main2}. In fact, we establish a result that is more general than Theorem \ref{main2}; see \Cref{1.6prf}. Before presenting the proof of the theorem, we provide some preliminary results.

\begin{chunk} \label{Yos} Let $R$ be a ring and let $M$ and $N$ be $R$-modules such that $\depth_{R_{\p}}(M_{\p}) \ge \depth (R_{\p})$ for all $\p\in \Spec(R)$. If $\pd_R(N)<\infty$, then $\Tor^R_i(M,N)=0$ for all $i\geq 1$; see \cite[5.3.1 and 5.3.6]{Gdimbook}. 
\end{chunk}

Next we outline some basic properties of semidualizing modules and refer the reader to \cite{smd} for further details:

\begin{chunk}\label{semibasic} Let $R$ be a ring. An $R$-module $C$ is called \emph{semidualizing} if $\Ext_R^i(C,C)=0$ for all $i\ge 1$ and the homothety map $R\to \Hom_R(C, C)$ is an isomorphism. For example, the $R$-module $R$ is semidualizing. Also, if $R$ is a Cohen-Macaulay ring with canonical module $\omega$, then $\omega$ is a semidualizing $R$-module. There are examples of Cohen-Macaulay rings $R$ that admit semidualizing modules which are distinct from $R$ and $\omega$.

Let $C$ be a semidualizing $R$-module. Then the following hold:

\begin{enumerate}[\rm(i)]
\item If  $\{\underline{x}\} \subseteq \fm$ is an $R$-regular sequence, then $C/\underline{x} C$ is a semiduaizing $R/\underline{x}R$-module.
\item If $C$ is semidualizing, then $C$ is faithful so that $\Supp_R(C)=\Spec(R)$.
\item If $C$ is semidualizing and $\p \in \Spec(R)$, then $C_{\p}$ is a semidualizing $R_{\p}$-module. 
\item If $R$ is local and $C$ is semidualizing, then $\depth_R(C)=\depth(R)$.
\item If $C$ is semidualizing, then $\depth_{R_{\p}}(C_{\p})= \depth(R_{\p})$ for all $\p\in \Spec(R)$.
\end{enumerate}
\end{chunk}


The following results from \cite{semi} are used in the proofs of Corollaries \ref{semi1}, \ref{semimain}, and \ref{1.6prf}.

\begin{chunk} \label{semifact} Let $R$ be a ring, $C$ be a semidualizing $R$-module, and let $T$ be an $R$-module.
\begin{enumerate}[\rm(i)]
\item If $\pd_R(T)<\infty$, then $\Hom_R(C, C\otimes_RT)\cong T$; see \cite[1.8(d) and 1.9(b)]{semi}. 
\item Set $X=\Hom_R(C,T)$. If $\pd_R(X)<\infty$, then $C\otimes_RX\cong T$ and $\Ext^i_R(C,T)=0$ for all $i\geq 1$; see \cite[1.8(a, b), 1.9(b) and 2.8(a)]{semi}.
\end{enumerate}
\end{chunk}

Next we recall a beautiful result of Sharp:

\begin{chunk} \label{Sharp} Let $R$ be a Cohen-Macaulay ring with canonical module $\omega_R$ and let $M$ be an $R$-module. If $\id_R(M)<\infty$, then $\pd_R\big(\Hom_R(\omega_R, M)\big)<\infty$. Also, if $\pd_R(M)<\infty$, then $\id_R(M\otimes_R \omega_R)<\infty$; see \cite[2.6 and 2.9]{Sharp71}
\end{chunk}

Recall that, given a module $M$ over a local ring $R$, $\mu_R(M)$ denotes the minimal number of generators of $M$. 

\begin{cor}\label{semi1} Let $R$ be a local ring, $C$ be a semidualizing $R$-module, and let $M$ and $N$ be nonzero $R$-modules. Assume $\pd_R(\Hom_R(C, M\otimes_R N))\le r <\infty$. If $r\geq 1$, assume further $\Ext_R^i(M,C)=0$ for all $i=1,\dots,r$. Then there is a surjective $R$-module homomorphism $M^{\oplus \mu_R(N)} \rightarrow C$.
\end{cor}  

\begin{proof} Set $X=\Hom_R(C, M\otimes_R N)$. As $\pd_R(X)<\infty$, we make use of \ref{semifact}(ii) with $T=M\otimes_RN$ and conclude that $C\otimes_R \Hom_R(C, M\otimes_R N)=C\otimes_RX\cong M\otimes_R N$. 

Note that there is a surjective $R$-module homomorphism $R^{\oplus u}\to N$, where $u=\mu_R(N)$. Tensoring this surjection by $M$, and using the fact that $C\otimes_RX\cong M\otimes_R N$, we obtain a surjective $R$-module homomorphism $M^{\oplus u}\to C\otimes_RX$. Therefore, $M^{\oplus u}$ maps surjectively onto $C$ in case $X$ is free. So, we may assume $r\geq 1$. As $\depth_{R_{\p}}(C_{\p})= \depth(R_{\p})$ for all $\p\in \Spec(R)$, it follows from \ref{Yos} that $\Tor^R_i(X, C)=0$ for all $i\ge 1$; see \ref{semibasic}(v). Hence, we conclude from  \Cref{injprop0} that there is a surjective $R$-module homomorphism $M^{\oplus u} \rightarrow C$. 
\end{proof}

The next result  is an application of \Cref{semi1}.

\begin{cor} \label{injprop} Let $R$ be a Cohen-Macaulay local ring with a canonical module $\omega_R$, and let $M$ and $N$ be nonzero $R$-modules. Assume $M$ is maximal Cohen-Macaulay and $\id_R(M\otimes_RN)<\infty$. Then there is a surjective $R$-module homomorphism $M^{\oplus \mu_R(N)} \rightarrow \omega_R$. 
\end{cor}

\begin{proof} As $\id_R(\omega)<\infty$, it follows from \ref{Sharp} that $\pd_R(\Hom_R(\omega_R, M\otimes_RN))<\infty$. Moreover, $\Ext^i_R(M,\omega_R)=0$ for all $i\geq 1$ because $M$ is maximal Cohen-Macaulay. Since $\omega_R$ is semidualizing, \Cref{semi1} yields the required surjective homomorphism.
\end{proof}

In passing, we recall the following basic facts as they are used in the proof of Theorem \ref{semimain}.

\begin{chunk} \label{AY} Let $R$ be a ring and let $M$ be an $R$-module.
\begin{enumerate}[\rm(i)]
\item The natural map $R \to \Hom_R(M,M)$ is injective if and only if $M$ is faithful over $R$. Hence, if $M$ is faithful over $R$, then $\widehat{M}$ is faithful  over $\widehat{R}$.
\item If $R$ is local and $N$ is an $R$-module such that $\Ext^i_R(M,N)=0$ for all $i\geq 1$, then $\depth_R\big(\Hom_R(M,N)\big)=\depth_R(N)$; see, for example, \cite[4.1]{ArY}.
\end{enumerate}
\end{chunk}


\begin{thm}\label{semimain} Let $R$ be a local ring and let $M$ and $N$ be nonzero $R$-modules. Assume there is an $R$-regular sequence $\{\underline{x}\} \subseteq \fm$ such that $M/\underline{x} M$ is not faithful over $R/\underline{x}R$. 
\begin{enumerate}[\rm(i)]
\item Let $C$ be a semidualizing $R$-module. Then $\pd_R\big(\Hom_R(C, M\otimes_R N)\big)=\infty$ provided that the following condition holds: If $\depth(R)>\depth_R(M\otimes_RN)$, we have that $\Ext_R^{i}(M,C)=0$ for all $i=1,\dots, \depth(R)-\depth_R(M\otimes_RN)$. 
\item If $M$ is maximal Cohen-Macaulay, then $\id_R(M\otimes_RN)=\infty$. 
\end{enumerate}
\end{thm}

\begin{proof} (i) Suppose there is a surjective $R$-module homomorphism $M^{\oplus n} \rightarrow C$ for some $n\geq 1$. This yields a surjective $R/\underline{x}R$-module homomorphism $(M/\underline{x}M)^{\oplus n} \rightarrow C/\underline{x}C$. Note that $C/\underline{x}C$ is faithful over $R/\underline{x}R$; see \ref{semibasic}(i, ii). This forces $M/\underline{x}M$ is faithful over $R/\underline{x}R$, which is a contradiction due to the hypothesis of the theorem. 

Set $t=\depth(R)$, $v=\depth_R(M\otimes_RN)$, and $X=\Hom_R(C, M\otimes_R N)$. Suppose $\pd_R(X)<\infty$ and seek a contradiction. By the argument above, it suffices to obtain a surjective $R$-module homomorphism $M^{\oplus n} \rightarrow C$ for some $n\geq 1$. Set $\pd_R(X)=r$. If $r=0$, that is, $X$ is free, we obtain such a surjection from Corollary \ref{semi1}. Hence, we may assume $r\geq 1$. As $X\neq 0$, the Auslander-Buchsbaum formula shows that $r=t-\depth_R(X)$.

As $\pd_R(X)<\infty$, we have that $\Ext_R^i(C,M\otimes_R N)=0$ for all $i\ge 1$; see \ref{semifact}(ii). Therefore, $\depth_R(X)=v$ so that $r=t-v$; see \ref{AY}(ii). We assume $\Ext_R^{i}(M,C)=0$ for all $i=1,\dots, r$, and hence Corollary \ref{semi1} yields a surjection $M^{\oplus u} \rightarrow C$, where $u=\mu_R(N)\geq 1$. This proves that $\pd_R(X)$ cannot be finite.

(ii) Assume $M$ is maximal Cohen-Macaulay, $\id_R(M\otimes_RN)<\infty$, and seek a contradiction. Note that $R$ is Cohen-Macaulay; see \cite[9.6.4]{BH}.
We may assume $R$ is complete with canonical module $\omega_R$; see \ref{AY}(i). Moreover, in view of \ref{Sharp}, we have that $\pd_R\big(\Hom_R(\omega, M\otimes_RN)\big)<\infty$. Furthermore, since $M$ is maximal Cohen-Macaulay, we have that $\Ext^i_R(M,\omega_R)=0$ for all $i\geq 1$. Since $\omega_R$ is semidualizing, we get a contradiction due to part (i) of the theorem. 
\end{proof}

\begin{cor}\label{1.6prf} Let $R$ be a local ring such that $\depth(R)=t$, $C$ be a semidualizing $R$-module, and let $M$ and $N$ be $R$-modules. Assume  
$M=\Omega_R L$ for some nonfree $R$-module $L$ with $\depth_R(L) \geq t$. If $t\geq 1$, assume further $\Ext_R^i(M,C)=0$ for all $i=1,\dots,t$. Then the following conditions are equivalent:
\begin{enumerate}[\rm(i)]
\item $N=0$
\item $\pd_R(M\otimes_R N)<\infty$.
\item $\pd_R\big(\Hom_R(C,M\otimes_R N)\big)<\infty$. 
\end{enumerate}
\end{cor}

\begin{proof} We assume part (iii) holds and proceed to show that part (i) holds. As $\depth_R(L) \ge t$, there is a maximal $R$-regular sequence $\{\underline{x}\}\subseteq \fm$ such that $\underline{x}$ is $L$-regular. Consider the minimal syzygy exact sequence $0\to M \to R^{\oplus \mu_R(L)}\to L \to 0$. Tensoring this sequence with $R/\underline{x}R$, we obtain the exact sequence $0\to M/\underline{x}M \to (R/\underline{x}R)^{\oplus \mu_R(L)}\to L/\underline{x}L\to 0$. Note that $\mu_R(L)=\mu_{R/\underline{x}R}(L/\underline{x}L)$. Therefore, $M/\underline{x}M \cong \Omega_{R/\underline{x}R}(L/\underline{x}L)$. This implies that $M/\underline{x}M\subseteq \fm (R/\underline{x}R)^{\oplus \mu(L)}$.

We have that $\depth(R/\underline{x} R)=0$. Hence, $\Soc(R/\underline{x}R)\neq 0$ and it annihilates $\fm(R/\underline{x} R)$. Consequently, $M/\underline{x} M$ is not faithful over $R/\underline{x} R$ because $M/\underline{x}M\subseteq \fm (R/\underline{x}R)^{\oplus \mu(L)}$. Since $\pd_R\big(\Hom_R(C,M\otimes_R N)\big)<\infty$ and $\Ext_R^i(M,C)=0$ for all $i=1,\dots,t$, Theorem \ref{semimain}(i) implies that $M=0$ or $N=0$.
As we assume $L$ is not free, it follows that $M\neq 0$. Thus $N=0$, and part (i) holds.

Next assume part (ii) holds. Set $T=M\otimes_R N$. Then $\Hom_R(C, C\otimes_RT)\cong T$; see \ref{semifact}(i).  Therefore, $T\cong \Hom_R(C, C\otimes_RT)=\Hom_R\big(C, M\otimes_R(C\otimes_R N)\big)$ has finite projective dimension. Now we make use of the fact that part (iii) implies part (i), with the modules $M$ and $C\otimes_RN$, and deduce that $C\otimes_R N=0$. As $C\neq 0$, we see that  $N=0$. So part (i) holds.
\end{proof}

Theorem \ref{main2} is subsumed by the following consequence of Corollary \ref{1.6prf}.

\begin{cor} \label{corend} Let $R$ be a Cohen-Macaulay local ring and let $M$ and $N$ be $R$-modules such that
$M=\Omega_R L$ for some nonfree maximal Cohen-Macaulay $R$-module $L$. Then the following conditions are equivalent:
\begin{enumerate}[\rm(i)]
\item $N=0$
\item $\pd_R(M\otimes_R N)<\infty$.
\item $\id_R(M\otimes_R N)<\infty$.  
\end{enumerate}
\end{cor}

\begin{proof} We may assume $R$ is complete with canonical module $\omega_R$. Note that $\id_R(M\otimes_R N)<\infty$ if and only if $\pd_R\big(\Hom_R(\omega_R,M\otimes_R N)\big)<\infty$; see \ref{Sharp}. Also, as $M$ is a maximal Cohen-Macaulay $R$-module, it follows that $\Ext_R^i(M,\omega_R)=0$ for all $i\ge 1$.  Hence, since $\omega_R$ is semidualizing, the claims hold due to  Corollary \ref{1.6prf}. 
\end{proof}

%

%

\subsection*{Some further corollaries} We finish this section by giving two corollaries that demonstrate how our results can be used. The first corollary we establish, namely Corollary \ref{c30}, yields an affirmative answer to Question \ref{soru1} when one of the modules in question is a finite direct sum of high syzygies of the residue field. First we recall:
 
 \begin{chunk} \label{DT} Let $R$ be a Cohen-Macaulay local ring with canonical module $\omega_R$. If a finite direct sum of copies of syzygy modules of the residue field of $R$ maps surjectively onto $\omega_R$, then $R$ is regular; see \cite[3.6]{GGP}. 
\end{chunk}

\begin{cor} \label{c30} Let $R$ be a $d$-dimensional Cohen-Macaulay local ring and let $N$ be an $R$-module. Set $M=\bigoplus\limits_{i=d}^{d+n}\big(\Omega^{i}_R k \big)^{\oplus a_i}$, for some integers $n\geq 0$ and $a_i\geq 1$. If $\pd_R(M\otimes_RN)<\infty$ or  $\id_R(M\otimes_RN)<\infty$, then $R$ is regular so that $\pd_R(M)<\infty$ and $\pd_R(N)<\infty$.
\end{cor}

\begin{proof} We may assume $R$ is complete with canonical module $\omega_R$. Note that $M$ is maximal Cohen-Macaulay. If $\id_R(M\otimes_RN)<\infty$, then the claim follows from Corollary \ref{injprop} and \ref{DT}. 

Assume $\pd_R(M\otimes_RN)<\infty$. Then $\id_R(M\otimes_RT)<\infty$, where $T=N\otimes_R\omega$; see \ref{Sharp}. In that case we use Corollary \ref{injprop} with the modules $M$ and $T$, and deduce from \ref{DT} that $R$ is regular.
\end{proof}

A maximal Cohen-Macaulay module $M$ over a Cohen-Macaulay local ring $R$ is said to be \emph{Ulrich} with respect to an $\fm$-primary ideal $I$ of $R$ provided that $M/IM$ is free over $R/I$ and the multiplicity $\e(I,M)$ of $M$ with respect to $I$ equals the length of $M/IM$; see, for example, \cite{Ulrich87}. 
An Ulrich module with respect to $\fm$ is simply called Ulrich. Ulrich modules are currently of significant research interest; they have been studied extensively, and examples of such modules are abundant in the literature. For example, if $R$ is a one-dimensional local domain, then $\fm^{i}$ is Ulrich with respect to $\fm$ for all $i\gg 0$; see \cite{GotoUlrich, Ulrich2} for details and further examples.

\begin{cor} \label{c37} Let $R$ be a Cohen-Macaulay local ring, $M$ be an Ulrich $R$-module with respect to an $\fm$-primary ideal $I$ of $R$, and let $N$ be nonzero $R$-module. If $\pd_R(M\otimes_RN)<\infty$, then $M$ is free, $\pd_R(N)<\infty$, and $\pd_R(I)<\infty$. Therefore, if $M$ is Ulrich, that is, if $I=\fm$, then $R$ is regular.
\end{cor}

\begin{proof} We can consider the faithfully flat extension $R \to \widehat{R[x]_{\fm[x]}}$, and hence assume $R$ is complete with canonical $\omega_R$ and infinite residue field; see, for example \cite[page 48]{Ei}. Then there exists a parameter ideal $\underline{x}=(x_1,\ldots, x_d)\subseteq I$ such that $IM=\underline{x}M$; see \cite[3.1]{GotoUlrich}. Thus the ideal $I/\underline{x}R$ of $R/\underline{x}R$ annihilates $M/\underline{x}M$. 

As $\pd_R(M\otimes_RN)<\infty$, it follows from \ref{Sharp} that $\id_R(M\otimes_RT)<\infty$, where $T=N\otimes_R \omega$. As $\underline{x}$ is a regular sequence on $R$, Theorem \ref{semimain}(ii) implies that $M/\underline{x}M$ must be faithful over $R/\underline{x}R$. Since $I/\underline{x}R\subseteq \Ann_{R/\underline{x}R}(M/\underline{x}M)$, we conclude that $I=\underline{x}R$. This implies that $\pd_R(I)<\infty$. 

We know $M/IM=M/\underline{x}M$ is free over $R/\underline{x}R$. As $\underline{x}$ is regular on both $R$ and $M$, we have that $\pd_R(M)=\pd_{R/\underline{x}R}(M/\underline{x}M)=0$, that is, $M$ is free.
\end{proof}


\appendix
\section{Proof of Theorem \ref{Roger} and some miscellaneous observations} \label{A}

This section contains a proof of Theorem \ref{Roger} as well as some observations yielding affirmative answers to Questions \ref{q1} and \ref{soru1} in some special cases. We should note that one can find various conditions in the literature under which special cases of these questions are true. Here, in this appendix, we point out only a few such results which motivate us and which are interesting for us.
\subsection*{A proof of Theorem \ref{Roger}}

In this subsection we give a proof of Theorem \ref{Roger} which is due to Roger Wiegand \cite {RW}. The proof relies upon the facts \ref{Free}, \ref{F1}, and \ref{obsR} stated next. 

\begin{chunk} \label{Free} Let $R$ be a commutative ring, and let $M$ and $N$ be $R$-modules. If $M\otimes_RN$ is nonzero and free, then $M$ and $N$ are both projective; see, for example \cite[3.4.7]{OldBook}. \qed
\end{chunk}

\begin{chunk} \label{F1} Let $R$ be a local ring and let $I$ be a proper ideal of $R$ such that $I$ contains a non zero-divisor on $R$. Then $I$ can be (minimally) generated by non zero-divisors on $R$. Here we justify this fact by giving a brief argument (taken from \cite{RW}).

We proceed by induction on the minimal number $v$ of generators required for $I$. Let $\{\fp_1, \ldots, \fp_t\}$ be the set of all associated primes of $R$. Choose $x_1 \in I-\big((\fm I)\cup (\fp_1 \cup \cdots \cup \fp_t)\big)$. Then $x_1$ is a non zero-divisor on $R$. If $v=1$, then $I=Rx_1$, and hence we are done. So we assume $v \geq 2$ and choose a minimal generating set for $I$, say $x_1, \ldots, x_v$. Let $J$ be the ideal of $R$ generated by $x_1, \ldots, x_{v-1}$. Then, since $J$ contains the non zero-divisor $x_1$, it follows by the induction hypothesis that $J$ is minimally generated by some elements $y_1, \ldots, y_{v-1}$, where each $y_i$ is a non zero-divisor on $R$. So $I$ is minimally generated by $y_1, \ldots, y_v$, where $y_v$ is an element in $I-\big((J+\fm I)\cup (\fp_1 \cup \cdots \cup \fp_t)\big)$. \qed
\end{chunk}

The fact recorded in \ref{F1} implies:

\begin{chunk} \label{obsR} Let $R$ be a local ring of positive depth. Then there is a sequence $\underline{x}=\{x_1, \ldots, x_n\} \subseteq \fm$ such that each $x_i$ is a non zero-divisor on $R$ and $k \cong \otimes_{i=1}^n(R/x_iR)$ for some positive integer $n$. Therefore $R$ is regular if the following condition holds: $\pd_R(R/\underline{x}R)<\infty$ whenever $\underline{x}$ is a sequence of elements in $\fm$, each of which is a non zero-divisor on $R$. \qed
\end{chunk}

We are now ready to give a proof of Theorem \ref{Roger}:

\begin{proof}[A proof of Theorem \ref{Roger}] (\cite{RW}) Note that, in view of \ref{Free}, it is enough to prove (i) $\Longrightarrow$ (ii) $\Longrightarrow$ (iii). The implication (i) $\Longrightarrow$ (ii) is due to the fact that, if $M$ and $N$ are $R$-modules, then $M\otimes_RN$ is a direct summand of $(M\oplus N)\otimes_R(M\oplus N)$. As \ref{obsR} establishes the implication (ii) $\Longrightarrow$ (iiii), the conclusions of the theorem hold.
\end{proof}

\begin{chunk} The argument used for the proof of Theorem \ref{Roger} also characterizes Gorenstein rings (respectively, complete intersection rings) by using the Gorenstein dimension \cite{AuBr} (respectively, the complete intersection dimension \cite{AGP}) instead of the projective dimension. For example a local ring $R$ of positive depth must be Gorenstein if $\gdim_R(M\otimes_RM)<\infty$ for each $R$-module $M$ with $\gdim_R(M)<\infty$.
\end{chunk}

\subsection*{An affirmative answer for Question \ref{q1}}

In this subsection we establish an observation advertised in the introduction, and obtain an affirmative answer for Question~\ref{q1} in a special case; see Proposition~\ref{cor3}. 

An $R$-module $N$ is said to satisfy $(\widetilde{S}_n)$ for some $n\geq 0$ if $\depth_{R_{\fp}}(N_{\fp})\geq \min\{n, \depth(R_{\fp})\}$ for all $\fp \in \Supp_R(N)$ (recall that $\depth_R(0)=\infty$ by convention). Note that, over Cohen-Macaulay rings, the condition $(\widetilde{S}_n)$ is nothing but a condition $(S_n)$ of Serre; see, for example \cite[page 3]{EG}. 

\begin{chunk} \label{cor1} Let $R$ be a local ring and let $M$ and $N$ be nonzero $R$-modules such that $\pd_R(M)<\infty$. Assume at least one of the following conditions holds:
\begin{enumerate}[\rm(i)]
\item $N$ satisfies $(\widetilde{S}_h)$, where $h=\pd_R(M)$.
\item $\G-dim_R(N)<\infty$ and $M$ satisfies $(\widetilde{S}_h)$, where $h=\G-dim_R(N)$.
\end{enumerate}
Then $\Tor_{i}^R(M, N) = 0$ for all $i\geq 1$. Therefore $\pd_R(M\otimes_RN)<\infty$ if and only if $\pd_R(N)<\infty$. 
\end{chunk}

\begin{proof} As $\pd_R(M)<\infty$, the following equality of Jorgensen \cite[2.2]{JAB} holds:
\begin{align}\notag{}
& \sup\{n\geq 0\mid \Tor_n^R(M,N)\not=0\}=  \\   & \notag{} \sup\{\depth(R_\p)-\depth_{R_{\fp}}(M_\fp)-\depth_{R_{\fp}}(N_\fp)\mid \p\in\Supp_R(M\otimes_RN)\}. 
\end{align}
Let $\fp\in \Supp_R(M\otimes_RN)$ and proceed to show that $\depth_{R_{\fp}}(M_\p)+\depth_{R_{\fp}}(N_\p)  \geq \depth(R_\fp)$; note that establishing this inequality is sufficient to conclude the vanishing of all $\Tor_{}^R(M, N)$ modules due to the equality of Jorgensen.

First assume the condition in part (i) holds. The claim follows if $h\geq \depth(R_\fp)$ since $N$ satisfies $(\widetilde{S}_h)$ so in that case $\depth_{R_{\p}} N_{\p} \ge \depth(R_\fp)$. If, on the other hand, $\depth(R_\fp)\geq h$, then $\depth_{R_{\p}} N_{\p}\ge h$, and the claim follows due to the following (in)equalities:
$$\depth_{R_{\fp}}(M_\p)+\depth_{R_{\fp}}(N_\p) \geq \depth_{R_{\fp}}(M_\p)+ \pd_R(M) \geq \depth_{R_{\fp}}(M_\p) + \pd_{R_{\fp}}(M_\p) = \depth(R_\fp).$$

Next assume the conditions in part (ii) hold.   Then it follows that $\gdim_{R_{\fp}}(N_\fp) \leq h$ and hence $\depth_{R_{\fp}}(N_\fp)=\depth(R_\fp)-\gdim_{R_{\fp}}(N_\fp)\ge \depth(R_\fp)-h$. So, in view of this inequality, since $M$ satisfies $(\widetilde{S}_h)$, we deduce that $\depth_{R_{\fp}}(M_\p)+\depth_{R_{\fp}}(N_\p)  \geq \depth(R_\fp)$. This proves the claim.
\end{proof}

\begin{chunk} We note that, if the condition in part (i) of \ref{cor1} holds, an alternative way to prove the vanishing of $\Tor_{i}^R(M, N)$ for all $i\geq 1$ is to utilize the following interesting results:
\begin{enumerate}[\rm(a)]
\item Let $R$ be a local ring and let $M$ be an $R$-module such that $\pd_R(M)=h<\infty$. Then there exists an $R$-regular sequence $\underline{x}=\{x_1, \ldots, x_h\} \subseteq \fm$ with the following property: $\Tor_{i}^R(M, N) = 0$ for all $i\geq 1$ whenever $N$ is an $R$-module such that $\underline{x}$ is a regular sequence on $N$; see \cite[2.5]{Dutta}.
\item If $R$ is a local ring and $N$ is an $R$-module satisfying $(\widetilde{S}_h)$ for some $h\geq 0$, then each $R$-regular sequence of length at most $h$ is also an $N$-regular sequence; see \cite[2.1]{aq}. 
\end{enumerate}
\end{chunk}

Our observation in \ref{cor1} can be compared with Example \ref{exRoger}(ii): there are $R$-modules $M$ and $N$ such that $\depth(R)-\depth_R(M)=\depth(R)-\depth_R(N)=1=\pd_R(M)=\pd_R(N)<\infty=\pd_R(M\otimes_RN)$, and $\G-dim_R(N)<\infty$, but neither $M$ nor $N$ satisfies $(\widetilde{S}_1)$.

Next we use \ref{cor1} and obtain an affirmative answer to Question \ref{q1} for a special case. 

\begin{prop} \label{cor3} Let $R$ be a local ring and let $M$ be an $R$-module which is locally free on the punctured spectrum of $R$. If $\pd_R(M) \leq \depth(R)/2$, then $\pd_R(M\otimes_RM)<\infty$. \qed
\end{prop}

\begin{proof} Set $d=\depth(R)$ and $h=\pd_R(M)$. Then $\depth_R(M)\geq d/2\geq h=\min\{d, h\}$. Hence, since $\depth_{R_{\fp}}(M_{\fp})=\depth(R_{\fp})$ for all $\fp \in \Supp_R(M)-\{\fm\}$, it follows that $M$ satisfies $(\widetilde{S}_h)$. Therefore, the claim follows from \ref{cor1}.
\end{proof}

\begin{eg} [{\cite[3.5]{CeD2}}] \label{exCD} Let $R=k[\![x,y,z]\!]/(xy-z^2)$ and let $I$ be the ideal of $R$ generated by $x$ and $y$. Then $R$ is a two-dimensional Cohen-Macaulay ring and $I$ is locally free on the punctured spectrum of $R$ such that $\pd_R(I) =1$. 
Hence one can use, for example, \ref{cor3} and conclude that $\pd_R(I\otimes_R I)<\infty$.
\end{eg}


\subsection*{Some affirmative answers for Question \ref{soru1}} In this subsection we record some observations giving affirmative answers to Question \ref{soru1}; see, for example, Proposition \ref{c32}. 

The first and the second parts of the next result are essentially due to Celikbas-Takahashi \cite{CTGlas} and Gheibi \cite{Mohsen}, respectively.

\begin{chunk} \label{c31} Let $R$ be a local ring and let $M$ and $N$ be nonzero $R$-modules. Assume $\pd_R(M \otimes_R N)<\infty$. 
\begin{enumerate}[\rm(i)]
\item If $M=\fm X$ for some nonzero $R$-module $X$, then $R$ is regular, and $\pd_R(M)<\infty$ and $\pd_R(N)<\infty$.
\item If $\id_R(M)<\infty$, then $R$ is Gorenstein and $\pd_R(M)<\infty$.
\end{enumerate}
 \end{chunk} 
  
\begin{proof}
The conclusion in part (i) follows from two facts: $\fm X \otimes_R N \cong \fm C$ for some nonzero $R$-module $C$ \cite[2.8]{CTGlas}. Moreover, if $\pd_R(\fm C)<\infty$; then $R$ is regular; see \cite[1.1]{LV}. 

For part (ii), note that we have a surjection $M^{\oplus r} \to M\otimes_R N$ for some $r\geq 1$. Hence, if $\id_R(M)<\infty$, \cite[4.1]{Mohsen} shows that $R$ is Gorenstein so that $\pd_R(M)<\infty$ since $\id_R(M)<\infty$; see \cite[2.2]{LV}.
\end{proof}

A special case of \ref{c31}(ii) is:

\begin{chunk} Let $R$ be a Cohen-Macaulay local ring with a canonical module $\omega$ and let $N$ be a nonzero $R$-module such that $\pd_R(\omega\otimes_RN)<\infty$. Then $R$ is Gorenstein and hence $\pd_R(N)<\infty$. 
\end{chunk}

The next observation is straightforward due to \ref{cor1} and \ref{c31}:

\begin{chunk} Let $R$ be a local ring and let $M$ and $N$ be nonzero $R$-modules such that $\pd_R(M \otimes_R N)<\infty$ and $\id_R(M)<\infty$. Then $\pd_R(M)<\infty$ and $\pd_R(N)<\infty$ if at least one of the following holds:
\begin{enumerate}[\rm(i)]
\item $N$ satisfies $(\widetilde{S}_h)$, where $h=\depth(R)-\depth_R(M)$.
\item $M$ satisfies $(\widetilde{S}_h)$, where $h=\depth(R)-\depth_R(N)$.
\end{enumerate}
 \end{chunk} 

The next result is essentially contained in \cite[4.7 and 4.8]{Bounds}; here it is reformulated in terms of the projective dimension. We add a brief argument for completeness.

\begin{prop} \label{c32} Let $R=\CC[\![x_0, \ldots, x_d]\!]/(f)$ be a simple singularity, where $0 \neq f \in (x_0, \ldots, x_d)^2$ and $d$ is a positive even integer. 
Let $M$ and $N$ be $R$-modules such that $\pd_R(M\otimes_R N)<\depth(R)$. If $M$ and $N$ are both locally free on the punctured spectrum of $R$, then $\pd_R(M)<\depth(R)$ and $\pd_R(N)<\depth(R)$. Therefore, if $M$ and $N$ are maximal Cohen-Macaulay, then $M$ and $N$ are free.
\end{prop}

\begin{proof} Assume $M$ and $N$ are both locally free on the punctured spectrum of $R$. Then $M\otimes_RN$ is torsion-free since $\depth_R(M\otimes_RN)\geq 1$. Hence $M\otimes_RN\cong \overline{M}\otimes_R N \cong M \otimes_R\overline{N}$, where $\overline{(-)}$ denotes the torsion-free part of the module in question; see \cite[1.1]{HW1}. So, in view of \cite[2.8 and 3.16]{Da1} and \cite[1.9]{HW2}, we can use an argument similar to \cite[2.11]{CGTT} and conclude that both $M$ and $N$ are torsion-free, and $\Tor_i^R(M,N)=0$ for all $i\geq 1$. The case where $M$ and $N$ are maximal Cohen-Macaulay follow similarly because $R$ is an isolated singularity, that is, $R_{\fp}$ is regular for each non-maximal prime ideal $\fp$ of $R$.
\end{proof}
 
\section*{acknowledgements}
The authors are grateful to Roger Wiegand for providing us with the write-up \cite{RW}, and for his valuable comments during the preparation of the manuscript. The authors also thank Hiroki Matsui for discussions about a previous version of the manuscript, especially about \ref{cor1}, and Yongwei Yao for discussions about Lemma \ref{YO} and \ref{F1}.

Part of this work was completed when Kobayashi visited West Virginia University in September 2022 and March 2023. 

\bibliography{a}

\def\cfudot#1{\ifmmode\setbox7\hbox{$\accent"5E#1$}\else \setbox7\hbox{\accent"5E#1}\penalty 10000\relax\fi\raise 1\ht7 \hbox{\raise.1ex\hbox to 1\wd7{\hss.\hss}}\penalty 10000 \hskip-1\wd7\penalty 10000\box7}
\providecommand{\bysame}{\leavevmode\hbox to3em{\hrulefill}\thinspace}
\providecommand{\MR}{\relax\ifhmode\unskip\space\fi MR }
\providecommand{\MRhref}[2]{%
  \href{http://www.ams.org/mathscinet-getitem?mr=#1}{#2}
}
\providecommand{\href}[2]{#2}
\begin{thebibliography}{10}

\bibitem{ArY}
Tokuji Araya and Yuji Yoshino, \emph{Remarks on a depth formula, a grade inequality and a conjecture of {A}uslander}, Comm. Algebra \textbf{26} (1998), no.~11, 3793--3806.

\bibitem{Au}
Maurice Auslander, \emph{Modules over unramified regular local rings}, Illinois J. Math. \textbf{5} (1961), 631--647.

\bibitem{AuBr}
Maurice Auslander and Mark Bridger, \emph{Stable module theory}, Mem. Amer. Math. Soc., No. 94, American Mathematical Society, Providence, R.I., 1969.

\bibitem{AB58}
Maurice Auslander and David~A. Buchsbaum, \emph{Codimension and multiplicity}, Ann. of Math. (2) \textbf{68} (1958), 625--657.

\bibitem{Av1}
Luchezar~L. Avramov, \emph{Modules of finite virtual projective dimension}, Invent. Math. \textbf{96} (1989), no.~1, 71--101.

\bibitem{Av2}
\bysame, \emph{Infinite free resolutions}, six lectures on commutative algebra ({B}ellaterra, 1996), Progr. Math., vol. 166, Birkh\"auser, Basel, 1998, pp.~1--118.

\bibitem{AFGD}
Luchezar~L. Avramov and Hans-Bj{\o}rn Foxby, \emph{Ring homomorphisms and finite {G}orenstein dimension}, Proc. London Math. Soc. \textbf{75} (1997), no.~2, 241--270.

\bibitem{AGP}
Luchezar~L. Avramov, Vesselin~N. Gasharov, and Irena~V. Peeva, \emph{Complete intersection dimension}, Inst. Hautes \'Etudes Sci. Publ. Math. (1997), no.~86, 67--114 (1998).

\bibitem{OldBook}
Stanis{\l}aw Balcerzyk and Tadeusz J\'{o}zefiak, \emph{Commutative {N}oetherian and {K}rull rings}, Ellis Horwood Series: Mathematics and its Applications, Ellis Horwood Ltd., Chichester; distributed by Prentice Hall, Inc., Englewood Cliffs, NJ, 1989, Translated from the Polish by Maciej Juniewicz and Sergiusz Kowalski.

\bibitem{Ulrich87}
Joseph~P. Brennan, J\"{u}rgen Herzog, and Bernd Ulrich, \emph{Maximally generated {C}ohen-{M}acaulay modules}, Math. Scand. \textbf{61} (1987), no.~2, 181--203.

\bibitem{BH}
Winfried Bruns and J{{\"u}}rgen Herzog, \emph{Cohen-{M}acaulay rings}, Cambridge Studies in Advanced Mathematics, vol.~39, Cambridge University Press, Cambridge, 1993.

\bibitem{Ce18}
Olgur Celikbas, \emph{On the vanishing of {H}ochster's theta invariant and a conjecture of {H}uneke and {W}iegand}, Pacific J. Math. \textbf{309} (2020), no.~1, 103 -- 144.

\bibitem{CeD2}
Olgur Celikbas and Hailong Dao, \emph{Necessary conditions for the depth formula over {C}ohen-{M}acaulay local rings}, J. Pure Appl. Algebra \textbf{218} (2014), no.~3, 522--530.

\bibitem{CGTT}
Olgur Celikbas, Shiro Goto, Ryo Takahashi, and Naoki Taniguchi, \emph{On the ideal case of a conjecture of {H}uneke and {W}iegand}, Proc. Edinb. Math. Soc. (2) \textbf{62} (2019), no.~3, 847--859.

\bibitem{Bounds}
Olgur Celikbas, Arash Sadeghi, and Ryo Takahashi, \emph{Bounds on depth of tensor products of modules}, J. Pure Appl. Algebra \textbf{219} (2015), no.~5, 1670--1684.

\bibitem{CET}
Olgur Celikbas and Ryo Takahashi, \emph{On the second rigidity theorem of {H}uneke and {W}iegand}, Proc. Amer. Math. Soc. \textbf{147} (2019), no.~7, 2733--2739.

\bibitem{CTGlas}
\bysame, \emph{Powers of the maximal ideal and vanishing of (co)homology}, Glasg. Math. J. \textbf{63} (2021), no.~1, 1--5.

\bibitem{Gdimbook}
Lars~Winther Christensen, \emph{Gorenstein dimensions}, Lecture Notes in Mathematics, vol. 1747, Springer-Verlag, Berlin, 2000.

\bibitem{Da1}
Hailong Dao, \emph{Decent intersection and {T}or-rigidity for modules over local hypersurfaces}, Trans. Amer. Math. Soc. \textbf{365} (2013), no.~6, 2803--2821.

\bibitem{Dutta}
Sankar Dutta, \emph{On modules of finite projective dimension}, Nagoya Math. Journal \textbf{219} (2015), 87--111.

\bibitem{Ei}
David Eisenbud, \emph{Homological algebra on a complete intersection, with an application to group representations}, Trans. Amer. Math. Soc. \textbf{260} (1980), no.~1, 35--64.

\bibitem{EG}
E.~Graham Evans and Phillip Griffith, \emph{Syzygies}, London Mathematical Society Lecture Note Series, vol. 106, Cambridge University Press, Cambridge, 1985.

\bibitem{Mohsen}
Mohsen Gheibi, \emph{Quasi-injective dimension}, J. Pure Appl. Algebra \textbf{228} (2024), no.~2, Paper No. 107468.

\bibitem{GGP}
Dipankar Ghosh, Anjan Gupta, and Tony~J. Puthenpurakal, \emph{Characterizations of regular local rings via syzygy modules of the residue field}, J. Commut. Algebra \textbf{10} (2018), no.~3, 327--337.

\bibitem{GotoUlrich}
Shiro Goto, Kazuho Ozeki, Ryo Takahashi, Kei-Ichi Watanabe, and Ken-Ichi Yoshida, \emph{Ulrich ideals and modules}, Mathematical Proceedings of the Cambridge Philosophical Society, vol. 156, Cambridge University Press, 2014, pp.~137--166.

\bibitem{Ulrich2}
Shiro Goto, Ryo Takahashi, and Naoki Taniguchi, \emph{Ulrich ideals and almost {G}orenstein rings}, Proc. Amer. Math. Soc. \textbf{144} (2016), no.~7, 2811--2823.

\bibitem{HW1}
Craig Huneke and Roger Wiegand, \emph{Tensor products of modules and the rigidity of {${\textrm Tor}$}}, Math. Ann. \textbf{299} (1994), no.~3, 449--476.

\bibitem{HW2}
\bysame, \emph{Tensor products of modules, rigidity and local cohomology}, Math. Scand. \textbf{81} (1997), no.~2, 161--183.

\bibitem{JAB}
David~A. Jorgensen, \emph{A generalization of the {A}uslander-{B}uchsbaum formula}, J. Pure Appl. Algebra \textbf{144} (1999), no.~2, 145--155.

\bibitem{KOT}
Kaito Kimura, Yuya Otake, and Ryo Takahashi, \emph{Vanishing of {E}xt modules over {C}ohen-{M}acaulay rings}, Proceedings of the 53rd {S}ymposium on {R}ing {T}heory and {R}epresentation {T}heory, Symp. Ring Theory Represent. Theory Organ. Comm., Osaka, 2022, pp.~138--142.

\bibitem{BLW}
Graham~J. Leuschke and Roger Wiegand, \emph{Cohen-{M}acaulay representations}, Mathematical Surveys and Monographs, vol. 181, American Mathematical Society, Providence, RI, 2012.

\bibitem{LV}
Gerson Levin and Wolmer~V. Vasconcelos, \emph{Homological dimensions and {M}acaulay rings}, Pacific J. Math. \textbf{25} (1968), 315--323.

\bibitem{aq}
Marie-Paule Malliavin, \emph{Condition $(a_q)$ de samuel et $q$-torsion}, Bulletin de la Soci{\'e}t{\'e} Math{\'e}matique de France \textbf{96} (1968), 193 -- 196.

\bibitem{Mi}
Claudia Miller, \emph{Complexity of tensor products of modules and a theorem of {H}uneke-{W}iegand}, Proc. Amer. Math. Soc. \textbf{126} (1998), no.~1, 53--60.

\bibitem{smd}
Keri Sather-Wagstaff, \emph{Semidualizing modules}, available at https://ssather.people.clemson.edu/DOCS/sdm.pdf, 2010.

\bibitem{Sharp71}
Rodney~Y. Sharp, \emph{Finitely generated modules of finite injective dimension over certain {C}ohen-{M}acaulay rings}, Proc. London Math. Soc. (3) \textbf{25} (1972), 303--328.

\bibitem{semi}
Ryo Takahashi and Diana White, \emph{Homological aspects of semidualizing modules}, Math. Scand. \textbf{106} (2010), no.~1, 5--22 (English).

\bibitem{RW}
Roger Wiegand, \emph{Tensor products of modules of finite projective dimension}, unpublished preprint, 2008.

\bibitem{Yo}
Yuji Yoshino, \emph{Cohen-{M}acaulay modules over {C}ohen-{M}acaulay rings}, London Mathematical Society Lecture Note Series, vol. 146, Cambridge University Press, Cambridge, 1990.

\end{thebibliography}
\bibliographystyle{amsplain}
\end{document}